\documentclass[a4paper]{article}

\usepackage[cp1251]{inputenc} 
\usepackage[T2A]{fontenc} 

\usepackage[english]{babel}
\usepackage{amsmath, amsfonts, amssymb, amsthm}
\usepackage{floatflt}
\usepackage{multicol}
\usepackage{amsfonts}
\usepackage{caption}
\usepackage{xcolor}

\usepackage{pdfpages} 
\usepackage[12pt]{extsizes} 

\usepackage{graphicx} 

\newtheorem*{theorem*}{Theorem 1}
\newtheorem*{theorem**}{Theorem 2}
\newtheorem{lemma}{Lemma}[section]

\newtheorem{remark}{Remark}
\newtheorem{cor}{Corollary}[section]

\numberwithin{equation}{section}

\DeclareCaptionLabelSeparator{dot}{. }
\title{Simple closed geodesics on regular tetrahedra in Lobachevsky space}
 \author{Alexander ~A.~Borisenko, ~Darya ~D.~Sukhorebska}

\begin{document}
\date{}
\maketitle

{\it Abstract}.
We obtained a complete classification of simple closed geodesics on regular tetrahedra in   Lobachevsky space.
Also, we evaluated the number of simple closed geodesics of length not greater than $L$ and 
found the asymptotic of this number as $L$ goes to infinity.


\section{Introduction}\label{s1}

A closed geodesic is called simple if this geodesic is not self-intersecting and does not go along itself.
In 1905 Poincare proposed the conjecture on the existence of three simple closed 
geodesics on a smooth convex surface in three-dimensional Euclidean space.
In 1917  J. Birkhoff proved that 
there exists at least one simple closed geodesic
 on a Riemannian manifold that is homeomorphic to a sphere of arbitrary dimension 
 \cite{Birk}.
In 1929 L. Lyusternik and L.  Shnirelman obtained that there exist
at least three simple closed geodesics 
on a compact  simply-connected 
two-dimensional Riemannian manifold 
 \cite{LustShnir}, 
\cite{LustShnir47}.
But the   proof by Lyusternik and Shnirelman contains substantial gaps.
 I. A. Taimanov  gives a complete proof of the theorem
 that on each smooth Riemannian manifold homeomorphic to the  two-dimentional sphere
  there exist at least three distinct simple closed geodesics \cite{Tai92}.

In 1951 L. Lyusternik and A.  Fet	 stated that there exists at least one closed geodesic 
  on any compact Riemannian manifold  \cite{Fet51}. 
In 1965 Fet improved this results. 
He proved that there exist at least two closed geodesics  on a compact Riemannian manifold
under the assumption that all closed geodesics  are non-degenerate   \cite{Fet65}. 
V. Klingenberg generalized this theorem.
 He showed that on compact Riemannian manifolds with a finite fundamental group
 there exist infinitely many  closed geodesics \cite {Kling78}.

Due to works of  Huber \cite{Huber59}, \cite{Huber61} 
 it is known that 
  on  complete closed two-dimensional   surfaces  of constant negative curvature 
 the number of closed geodesics of length   not greater than $L$ 
 is of order $e^{L}/L$ as $L \rightarrow +\infty$.
 Ya. Sinai \cite{Sin66} and   G. Margulis \cite{Marg69} and others 
 generalized   this estimation for
  compact $ n $-dimensional manifolds of non-constant negative curvature.
Igor Rivin studied the growth of $N(L)$, the number of simple closed
geodesics of length bounded above by $L$ on a hyperbolic surface of
genus $g$ with $n$ points at infinity.
He proved \cite{Rivin} that    there exist constants $c_1$ and $c_2$  such that
 \begin{equation}\label{Rivin}
 c_1 L^{6g-6+2n}\le N(L) \le c_2 L^{6g-6+2n}.
 \end{equation}
 M. Mirzakhani \cite {Mirz08} showed that the constants $ c_1 $ and $ c_2 $ are equal and found its  value.

Geodesics on the non-smooth surfaces   are also investigated.
A. Cotton and others   described all simple closed geodesics 
on a cube, regular tetrahedron, regular octahedron   and rectangular prism \cite{Cot05}.
K. Lawson and others obtain a complete classification of  simple closed geodesics on   eight
convex deltahedra such that the faces of these deltahedra are regular triangles  \cite{Law10}.

D. Fuchs and K. Fuchs supplemented and systematized results about closed geodesics on a regular polyhedra 
in three-dimensional Euclidean space \cite{FucFuc07}, \cite{Fuc09}. 
V. Protasov  obtained  сonditions for the existence of closed geodesics on an arbitrary simplex
and gave the estimate  for the number 
of closed geodesics on a simplex  depending on 
the largest deviation from $ \pi $ of the sum of the plane angles at a vertex of the simplex \cite{Pro07}.

The aim of this paper is to describe all simple closed geodesics on   regular tetrahedra 
 in   three-dimensional  Lobachevsky space.
In Euclidean space since the  Gaussian curvature of   faces of a tetrahedron   is equal to zero,
 it follows that the   curvature  of the tetrahedron is  concentrated into its vertices.
In  Lobachevsky space  the   faces of a tetrahedron  have a constant  negative curvature.
 Thus the  curvature  of such tetrahedron defined not only by its vertices but also by its faces.
Furthermore, all regular tetrahedra in Euclidean space are homothetic to each others.
In   Lobachevsky space  the intrinsic geometry of a regular tetrahedron 
depends on the plane angles of the faces of this tetrahedron.
Regular euclidean triangles form a regular tiling of the euclidean plane.
 From this  it is easy to proof  
 the full  classification of  closed geodesics on  regular tetrahedrons in Euclidean space.
In general it is impossible to make a triangular tiling of the Lobachevsky plane by regular triangles.

It is true 

\begin{theorem*}
On a regular tetrahedron in Lobachevsky space for any coprime integers $(p, q)$, $0 \le p<q$, 
there exists  unique, up to the rigid motion of the tetrahedron, simple closed geodesic  of type $(p,q)$.
The geodesics of type $(p,q)$ exhaust all simple closed geodesics on a regular tetrahedron in  Lobachevsky space.
\end{theorem*}
The  simple closed geodesic  of type $(p,q)$ has
$p$ points   on   each of two opposite edges of the tetrahedron,
$q$ points   on   each of another two opposite edges,
and there are $(p+q)$ points   on    each edges of the third pair of opposite edges.

\begin{theorem**}
Let $N(L, \alpha) $ be a number of simple closed geodesics of length  not greater than $L$ 
on a regular tetraedron  with  plane angles of the faces equal to $\alpha$ in Lobachevsky space.
Then there exists a function  $c(\alpha) $ such that 
\begin{equation}\label{NLalpha}
 N(L, \alpha) =
c(\alpha) L^2 +O(L\ln L), \notag
\end{equation}
where $O(L\ln L) \le CL\ln L$ as $L \rightarrow +\infty$, 
 $c(\alpha)>0$ when $0<\alpha<\frac{\pi}{3}$ and
\begin{equation}\label{cpi3}
\lim_{\alpha \rightarrow \frac{\pi}{3}} c(\alpha) = +\infty;
\; \; \; \; 
\lim_{\alpha \rightarrow 0} c(\alpha) = \frac{27}{32  (\ln 3)^2 \pi^2 }. \notag
\end{equation}
\end{theorem**}

\section{Definitions}\label{construction}

A curve $\gamma$ is called a {\it geodesic} if any sufficiently small subarc of $\gamma$ 
realizes a shortest path between endpoints of this subarc.
On a convex polyhedron a geodesic satisfies the following properties \cite {Alek50}: \\
1) within every faces the geodesic is a straight line segment;\\
2) the geodesic traverses an interior point of an edge “without refraction”, 
i.e., the angles that it forms with the edge are equal from the two sides;\\
3) the geodesic does not pass through a vertex of a convex polyhedron.
 
\begin{remark}
By a straight line segment we mean a segment of a geodesic in the space of constant curvature.
\end{remark}

Assume a geodesic on a convex polyhedron starts at a point $X$ of an edge.
Then it goes into a face of the polyhedron,
intersects another edge of the same face at a point $Y$ and then passes to the next face, and so on.
Draw these faces in the plane along the geodesic.
In this way we obtain the polygon that is called the {\it development} of a polyhedron along the geodesic.
The geodesic becomes a straight line on it.
If a geodesic is closed, then this straight line arrives to the initial edge at the point $X'$ at the same position as $X$.

We assume that the  Gaussian curvature of    {\it  Lobachevsky space} ({\it hyperbolic space})  is equal to $-1$.
 By a {\it  regular tetrahedron} in   three-dimensional  Lobachevsky space 
we understand a closed convex polyhedron  
formed by four regular triangles.
These triangles are called {\it faces} and they  form   regular trihedral angles 
in the vertices of the tetrahedron.
Note that the plane angle $\alpha $ of the face  satisfies the inequality $0< \alpha<\frac{\pi}{3}$.
Up to a rigid motion in Lobachevsky space there exists a unique tetrahedron with  a given plane angle of its face.
The  length $a$ of every its edge is equal to
\begin{equation}\label{a}
a=\text{arcosh} \left(  \frac{\cos\alpha}{1-\cos\alpha}  \right).
\end{equation}
The following auxiliary formulas  follows from (\ref{a}):
\begin{equation}\label{tha}
 \tanh a=\frac{  \sqrt{2\cos \alpha -1}  }{ \cos\alpha };
\end{equation}
\begin{equation}\label{cha2}
 \cosh \frac{a}{2}=\frac{1}{ 2\sin \frac{\alpha}{2}  };
\end{equation}
\begin{equation}\label{sha2}
\sinh  \frac{a}{2}=\frac{ \sqrt{2\cos \alpha -1} } { 2\sin \frac{\alpha}{2}  }.
\end{equation}

Consider  the  Cayley–Klein model of   hyperbolic space. 
In this model points are represented by the points in the interior of the unit ball
\cite{Cannon97}.
Geodesics in this model are the chords of the ball.
Assume that the center  of the circumscribed sphere of a regular tetrahedron coincides 
with the center of the model.
Then the regular tetrahedron in   Lobachevsky space is represented by a regular tetrahedron in   Euclidean space.

\section{Necessary conditions for the simplicity of a geodesic  on a regular tetrahedron  in Lobachevsky space}

We start this section with formulation some useful lemmas.

\begin{lemma}\label{convexdist} \textnormal { \cite{Burag94} }.
Let $M$ be a complete simply connected Riemannian manifold of nonpositive curvature. 
Let $\gamma_1(s): [0, 1]\rightarrow M$ and  $\gamma_2(s):[0, 1]\rightarrow M$ be geodesics.
Let  $\sigma_s(t):[0,1]\rightarrow M$  be a geodesic such that  $\sigma_s(0)=\gamma_1(s)$, $\sigma_s(1)=\gamma_2(s)$, and
let $\rho(s)$ be the length of $\sigma_s$.
Then $\rho(s)$ is a convex function.
\end{lemma}

\begin{cor} \label{convexdistcor}\textnormal { \cite{Burag94} }. 
Let  $\gamma(s)$  be a  geodesic in a complete simply connected Riemannian manifold $M$ of nonpositive curvature.
Then the function $\rho(s)$ of a distance
from the fixed point $p$ on  $M$ to the points on $\gamma(s)$ is a convex function.
\end{cor}

\begin{lemma}\textnormal{\cite{Pras04}}.
Let $ABC$ be a right triangle $\left(\angle C = \frac{\pi}{2} \right)$   in  Lobachevsky space.
Then
\end{lemma}
\begin{equation}\label{Pifagor}
\cosh |AB| = \cosh |CA| \cosh |CB| \; (\textit{the hyperbolic  Pythagorean theorem});
\end{equation}
\begin{equation}\label{ththcos}
\tanh |AC| = \tanh |AB| \cos \angle A;
\end{equation}
\begin{equation}\label{thshtan}
\tanh |CB| = \sinh |AC| \tan \angle A;
\end{equation}
\begin{equation}\label{shshsin}
\sinh |CB| =\sinh |AB| \tan \angle A.
\end{equation}

Consider a regular tetrahedron with the plane angle $\alpha$ of its face.
Denote by $h$  the altitude of its face.
Using (\ref{ththcos}) and (\ref{tha}), we get 
\begin{equation}\label{h}
\tanh h =\tanh a \cos \frac{\alpha}{2} = \cos \frac{\alpha}{2} \frac{\sqrt{2\cos \alpha - 1}}{\cos \alpha}.
\end{equation}

\begin{lemma}
If  a geodesic successively intersects three edges sharing a common vertex
on a regular tetrahedron   in  Lobachevsky space
    and one of these edges twice,
   then this geodesic is self-intersecting.
   \end{lemma}
\begin{proof}

Let $A_1A_2A_3A_4$ be a regular tetrahedron in  Lobachevsky space.
Suppose the geodesic $\gamma$ successively intersects $A_4A_1$, $A_4A_2$ and $A_4A_3$  
at the points $X_1$, $X_2$, $X_3$ and then intersects $A_4A_1$ at the point $Y_1$.
If $Y_1$ coincides with $X_1$, then this point is a self-intersection point of $\gamma$.
Suppose that the length of $A_4X_1$ is  less than the  length of $A_4Y_1$.

Consider the trihedral angle at the vertex $A_4$. Cut this angle along the edge $A_4A_1$ and develop it to the plane.
We obtain the convex polygon $A_1A_4A'_1A_3A_2$.
The part $X_1X_2X_3Y_1$ of the geodesic is the straight line segment on the development.

Let $\rho(A_4, X)$ be the distance function between the vertex $A_4$ and a point $X$ on $\gamma$.
From Corollary \ref{convexdistcor} it follows that the function $\rho(A_4, X)$ is a convex function.
This function attains its minimum at the point $H_0$ such that $A_4H_0$ is orthogonal to $\gamma$ and 
$\angle H_0A_4Y_1 > \frac{3\alpha}{2}$.

\begin{figure}[h]
\begin{center}
\begin{minipage}[h]{0.4\linewidth}
\center{\includegraphics[width=50mm]{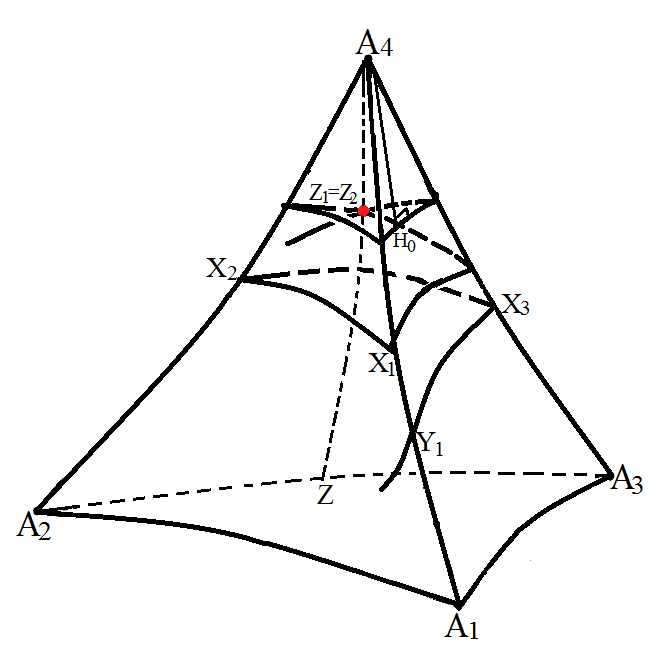} }
\caption{  }
\label{A1A4A1v1}
\end{minipage}
\hfill 
\begin{minipage}[h]{0.5\linewidth}
\center{\includegraphics[width=50mm]{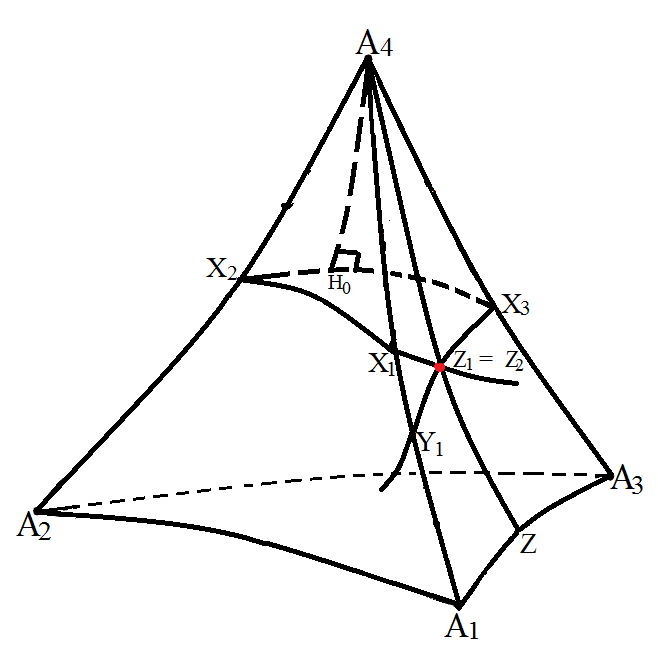} }
\caption{}
\label{A1A4A1v2}
\end{minipage}

\end{center}
\end{figure}

On the segment $H_0Y_1$ mark the point $Z_1$ such that $\angle H_0A_4Z_1 = \frac{3\alpha}{2}$.
Then on the part of $\gamma$ started from the point $H_0$ in the direction opposite to $H_0Y_1$
 mark the point $Z_2$ such that $\angle H_0A_4Z_2 = \frac{3\alpha}{2}$.
The point $Z_2$ also lies in a face incident to the vertex $A_4$. 

Since $\angle H_0A_4Z_1 = \angle H_0A_4Z_2 = \frac{3\alpha}{2}$, 
it follows that the  segments $A_4Z_1$ and $A_4Z_2$ belong to the   segment $A_4Z$ 
opposite to $A_4H_0$ on the tetrahedron.
Furthermore, the triangles $Z_1A_4H_0$ and $Z_2A_4H_0$ are equal.
It follows that on the tetrahedron the points $Z_1$ and $Z_2$ correspond to the same point on  $A_4Z$.
This point is the self-intersection point of the geodesic $\gamma$ (Figure \ref{A1A4A1v1}, Figure \ref{A1A4A1v2}).
\end{proof}

\begin{lemma}\label{nesdist}
Let $d$ be the smallest distance between the set of vertices of a regular tetrahedron
  in  Lobachevsky space 
   and a simple closed geodesic on this tetrahedron.
Then
\begin{equation}\label{distvertex}
\tanh d>\cos\frac{3\alpha}{2}\cos\frac{\alpha}{2} \frac{\sqrt{2\cos\alpha-1}}{\cos\alpha},
 \end{equation}
where $\alpha$ is the plane angle  of a face of the tetrahedron.
\end{lemma}

\begin{proof} 
Consider a regular tetrahedron $A_1A_2A_4A_3$   in Lobachevsky space
 and a simple closed geodesic $\gamma$  on this tetrahedron.
Assume that the distance between vertex  $A_4$ and $\gamma$
 is the  smallest distance from  the set of vertices of a  tetrahedron to  $\gamma$.
Construct the  line segment $A_4H$ that is orthogonal to $\gamma$ at the point $H_0$.
Assume  $A_4H$ belongs to the face $A_2A_4A_3$.
Denote by $\beta$ the angle between $A_4H $ and $A_4A_2$. 
Without loss of generality we assume that $0 \le \beta \le \frac{\alpha}{2}$. 

\begin{figure}[h]
\begin{center}
\begin{minipage}[h]{0.4\linewidth}
\includegraphics[width=50mm]{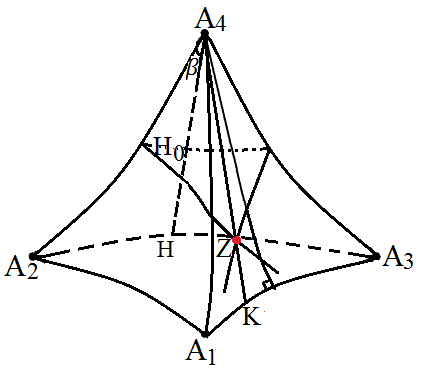}
\caption{  }
\label{nescond1}
\end{minipage}
\begin{minipage}[h]{0.5\linewidth}
\includegraphics[width=70mm]{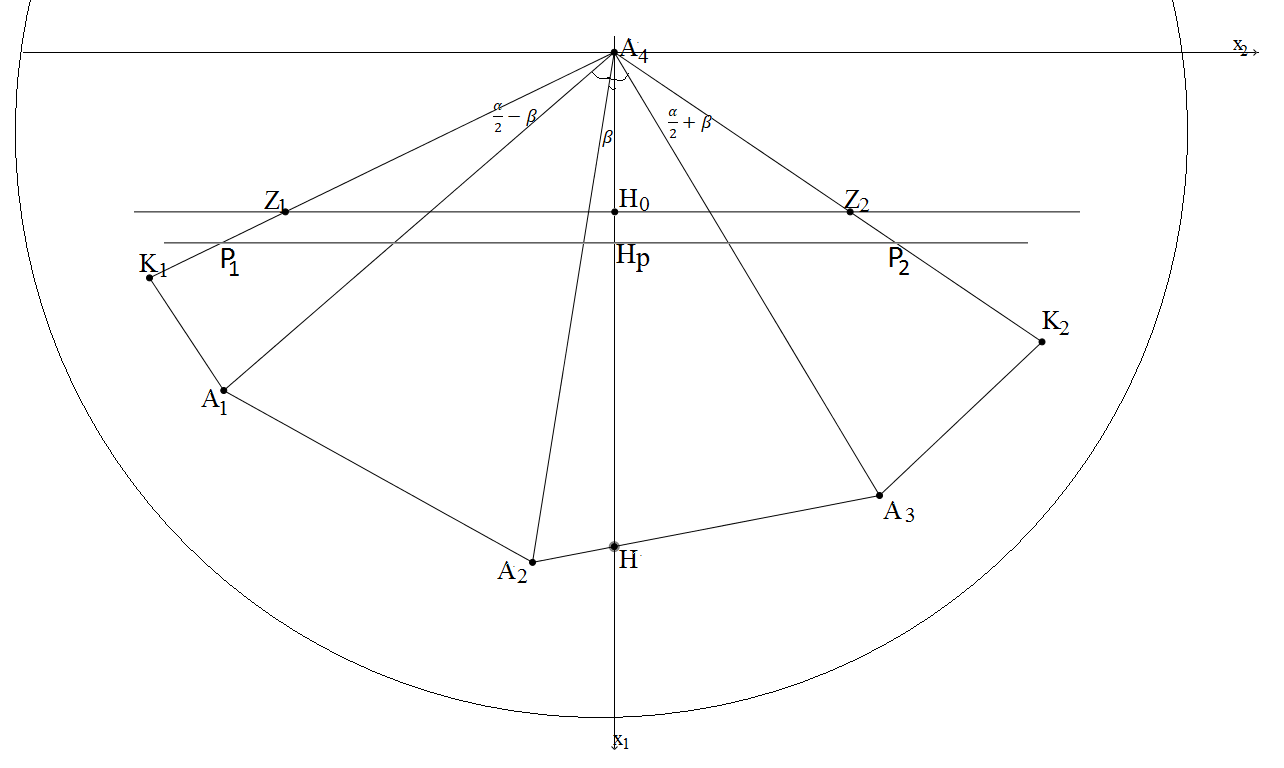} 
\caption{ }
\label{KAK}
\end{minipage}
\end{center}
\end{figure}

Construct the segment $A_4K$ such that
the plane angle between $A_4K$ and $A_4H$ equals  $\frac{3\alpha}{2}$ (Figure \ref{nescond1}). 
It follows that  $A_4K$ lies  in the face $A_1A_4A_3$ and 
the angle between $A_4K$ and $A_4A_1$ equals  $(\frac{\alpha}{2}-\beta)$ on this face.
Note that if $\beta=\frac{\alpha}{2}$, then $A_4K$ coincides with $A_4A_1$.
If $\beta = 0$, then $A_4K$ coincides with the altitude of the face and has the smallest length.

Cut the trihedral angle at the vertex $A_4$ by the line $A_4K$ and develop it to the plane.
We obtain the convex polygon $K_1A_4K_2A_3A_2A_1$.
Consider  the  Cayley–Klein model of the hyperbolic plane.
Assume that  the vertex $A_4$ of the development  coincide with the centre of the model (Figure \ref{KAK}).
The angle $K_1A_4K_2$ equals $3\alpha$. 
The segment $A_4H$ corresponds to the bisector of the angle $K_1A_4K_2$.
The geodesic $\gamma$ is a straight line that is ortogonal to $A_4H$ at the point $H_0$. 

On the lines  $A_4K_1$ and $A_4K_2$ mark the points $P_1$ and $P_2$  respectively so that 
the length of $A_4P_1$  and the length of $A_4P_2$ are equal to the length of the altitude $h$ of the face (Figure \ref{KAK}).
The line segment  $P_1P_2$ is ortogonal to $A_4H$ at the point $H_p$.
From  equation  (\ref{ththcos}) applied to  the triangle $A_4H_pP_1$  we obtain
\begin{equation}\label{A4Hp}
\tanh |A_4H_p| = \cos\frac{3\alpha}{2}\tanh  h.
\end{equation}

If $ d \le |A_4H_p|$, then $\gamma$ lies above the segment $P_1P_2$.
It follows that $\gamma$ intersects the  lines $A_4K_1$ and $A_4K_2$ at the points $Z_1$ and $Z_2$ respectively.

The segments $A_4K_1$ and $A_4K_2$ correspond to the segment $A_4K$ on the tetrahedron.
Then the points $Z'_1$ and $Z'_2$ are the same point $Z$ on the segment $A_4K$.
We obtain that the point $Z$ is the self-intersection point of the geodesic $\gamma$.

Thus we get that $ d >  |A_4H_p|$  is a necessary condition for
 the geodesic $\gamma$ to have no points of self-intersection on the regular tetrahedron   in  Lobachevsky space.
From  (\ref{A4Hp}) it follows that 
\begin{equation}\label{thdcosthh}
\tanh d > \cos\frac{3\alpha}{2}\tanh h.
\end{equation}

Using (\ref{h}) and  (\ref{thdcosthh}), we obtain inequality  (\ref{distvertex}). 
\end{proof}

\begin{cor}\label{nesdistcor}
Let $d$ be the smallest distance between the set of vertices of a regular tetrahedron
  in Lobachevsky space 
   and a simple closed geodesic on this tetrahedron.
Then
\begin{equation}\label{distvertex2}
d> \frac{1}{2} \ln \left( \frac{  \sqrt{2\pi^3} + \left( \pi- 3\alpha \right)^{\frac{3}{2}} }
{   \sqrt{2\pi^3} - \left( \pi- 3\alpha \right)^{\frac{3}{2}}    }   \right),
 \end{equation}
where $\alpha$ is the plane angle  of a face of the tetrahedron.
\end{cor}

\begin{proof}
It is known that
\begin{equation}\label{siny}
\sin y> \frac{2}{\pi}y   \;\;for\;\;   0<y<\frac{\pi}{2}.
 \end{equation}

Consider  the function $\sqrt{2\cos\alpha-1}$. We get 
 \begin{equation}
2\cos\alpha-1 = 2\cos\alpha-2\cos \frac{\pi}{3} = 
4\sin \left( \frac{\pi}{6}-\frac{\alpha}{2} \right)  \sin \left( \frac{\pi}{6}+\frac{\alpha}{2}  \right). \notag
 \end{equation}
 On the interval $ (0, \frac{\pi}{3})$ the function  $\sin(\frac{\pi}{6}+\frac{\alpha}{2})$ increases.
At $\alpha = 0$ this function equals $\frac{1}{2}$.
It follows that $\sin(\frac{\pi}{6}+\frac{\alpha}{2})>\frac{1}{2}$ when $\alpha \in (0, \frac{\pi}{3})$.
The function $\sin(\frac{\pi}{6}-\frac{\alpha}{2})$ increases on the interval $ (0, \frac{\pi}{3})$.
From (\ref{siny}) we get $\sin(\frac{\pi}{6}-\frac{\alpha}{2})>\frac{1}{\pi} \left( \frac{\pi}{3}-\alpha \right) $.
We obtain
 \begin{equation}\label{root}
\sqrt{2\cos\alpha-1} > \sqrt{ \frac{2}{3\pi} \left( \pi-3\alpha \right) }.
\end{equation}
 
 Consider $\cos \frac{3\alpha}{2}$.
 Since $0<\alpha<\frac{\pi}{3}$,  it follows that  $0< \frac{3\alpha}{2}<\frac{\pi}{2}$.
 On the interval $0<y<\frac{\pi}{2}$ the function $\cos y$ decreases.
We have the estimate
 \begin{equation}\label{cosy}
\cos y>1- \frac{2}{\pi}y,   \;\;\;\;     0<y<\frac{\pi}{2}. \notag
 \end{equation}
From this  estimate we obtain
 \begin{equation}\label{cos3alpha2}
\cos \frac{3\alpha}{2} > \frac{1}{\pi} \left( \pi-3\alpha \right).
 \end{equation}

  On the interval $0<\alpha<\frac{\pi}{3}$ the function  $\cos \frac{\alpha}{2}$ decreases.
  We  have $\cos \frac{\alpha}{2} > \frac{\sqrt{3}}{2}$ when $0<\alpha<\frac{\pi}{3}$.
 
From the computation above and (\ref{root}), (\ref{cos3alpha2}) it follows that
 \begin{equation}\label{vspomd}
\tanh d> \frac{1}{\sqrt{2\pi^3}} \left( \pi- 3\alpha \right)^{\frac{3}{2}}. 
 \end{equation}
 
 From  inequality  (\ref{vspomd}) we obtain   inequality  (\ref{distvertex2}).

\end{proof}

\section{ Properties of closed geodesics on regular tetrahedra in  Euclidean space}

Consider a regular tetrahedron $A_1A_2A_3A_4$ in Euclidean space (Figure \ref{ABCD}).
Assume the geodesic $\gamma$ starts at the point $X$ in the edge $A_1A_2$ and then goes to the face $A_1A_2A_4$. 
The development of a regular tetrahedron along the geodesic is a part of the regular triangular tiling of the Euclidean plane.
Label the vertices of this tiling according to the tetrahedron's vertices. 
It other words, for any development of the tetrahedron the labeling of its vertices 
corresponds to the labeling of the vertices of the tiling (Figure \ref{ABCD_dev}). 
 
 \begin{figure}[h]
\begin{center}
\begin{minipage}[h]{0.4\linewidth}
\includegraphics[width=45mm]{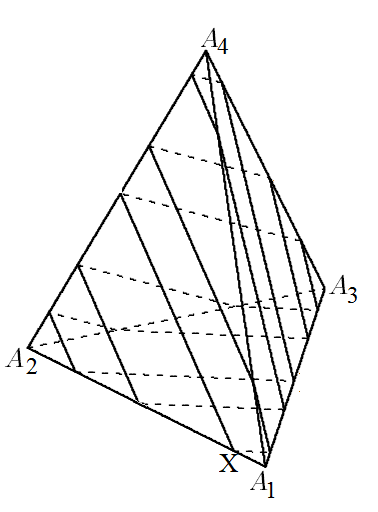}
\caption{ }
\label{ABCD}
\end{minipage}
\begin{minipage}[h]{0.5\linewidth}
\includegraphics[width=85mm]{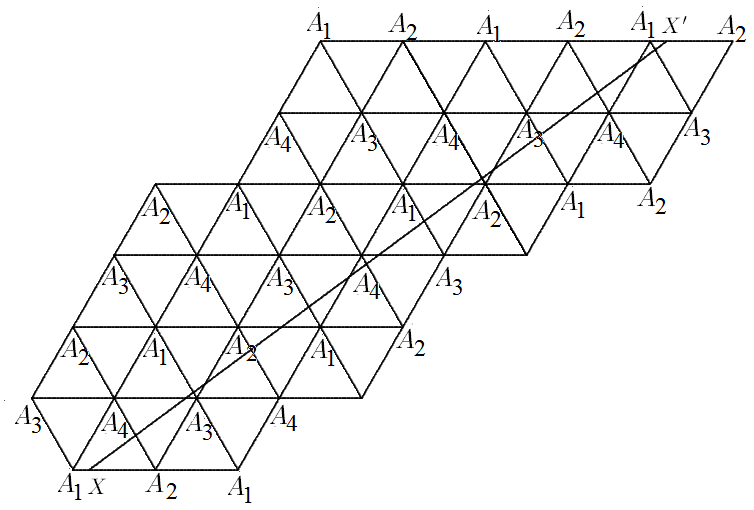}
\caption{ }
\label{ABCD_dev}
\end{minipage}
\end{center}
\end{figure}

To construct a closed geodesic, it is sufficient to choose two identically 
oriented edges in the tiling, for example $A_1A_2$, not on the same line, to mark  points
$X$ and $X'$  on these edges at the same distance from   $A_1$ and to join $X$ and $X'$ by the line segment.
Note that the points $X$ and $X'$  are such that the line $XX'$ does not pass through any vertex of the tiling.
The line segment  $XX'$  realizes a closed geodesic on the tetrahedron, and every closed
geodesic can be obtained in this way \cite{FucFuc07}.

Since the segments of the geodesic  within every
face are parallel to each other, it follows that all closed geodesic on a regular tetrahedron are non-self-intersecting 
(Figure \ref{ABCD}).

Two geodesics are called {\it equivalent} if they pass through the same edges in the same order on the polyhedron.

\begin{lemma}\label{centreqv}
Let $\gamma$ be a simple closed geodesic on a regular tetrahedron in  Euclidean space.
Then there exists a simple closed geodesic $\gamma_0$ equivalent to $\gamma$ such that
$\gamma_0$ passes through the midpoints of two pairs of   opposite edges on the tetrahedron. 
\end{lemma}

\begin{proof}

Consider a triangular tiling of the Euclidean plane 
and label the vertices of the tilling according to the vertices of a regular tetrahedron (Figure \ref{ABCD_dev}). 
Assume that the Cartesian coordinate system has the origin at the vertex $A_1$ 
and the $x$-axis is going  along the edge $A_1A_2$.
Denote by $(x_v, y_v)$ the coordinates of the vertices.
The vertices $A_1$ and $A_2$ belong to the line $y_v=2k\frac{\sqrt{3}}{2}$, 
and their first coordinate is $x_v=l$ $( k, l \in \mathbb{Z} )$.
The vertices $A_3$ and $A_4$ belong to the line $y_v=(2k+1)\frac{\sqrt{3}}{2}$,
and their first coordinate is $x_v=l+\frac{1}{2}$ $( k, l \in \mathbb{Z} )$.
 
Suppose a simple closed geodesic $\gamma$ starts at the point $X$ with  coordinates $(\mu, 0)$, where $0<\mu<1$.
Then the endpoint $X'$ of $\gamma$ has the coordinates $(\mu+q+2p, q\sqrt3)$, 
where $p, q$ are nonnegative integer numbers.
The line $XX'$ is $y=\frac{q\sqrt3}{q+2p} (x-\mu) $.
Note that if $p,q$ are coprime  integers, then the geodesic does not go along itself \cite{FucFuc07}.

First we proof that if geodesic passes through the midpoint of the one edge, then it passes 
through the midpoints of two pairs of the opposite edges.
Assume that a closed geodesic $\gamma_0$ passes through the midpoint of the edge $A_1A_2$.
Then the equation of $\gamma_0$ is 
\begin{equation}\label{gamma0}
y= \frac{q\sqrt3}{q+2p} (x-\frac{1}{2}).
\end{equation}

Substituting the coordinates of the points $A_3$ and $A_4$ to  equation (\ref{gamma0}), 
we get 
\begin{equation}\label{condpqkl}
q(2l-2k-1)=2p(2k+1).
\end{equation}

If $q$ is even then there exist $k$ and $l$ satisfying   equation (\ref{condpqkl}).
It follows that there exist the vertex of the tiling such that $\gamma_0$ passes through this vertex.
It contradicts the properties of $\gamma_0$, therefore $q$ is an odd integer.

The points $X_1$ and $X'_1$ with   coordinates $(\frac{1}{2}, 0)$ and $(q+2p+\frac{1}{2}, q\sqrt3)$ 
satisfy  equation (\ref{gamma0}).
These points are the middle point of the edge $A_1A_2$ on the tetrahedron.
Suppose that the point $X_2$ is the midpoint of $X_1X'_1$.
Then the coordinates of $X_2$ are 
$\left( \frac{q}{2}+p+\frac{1}{2}, \frac{q}{2}\sqrt3 \right)$.
Substituting $q=2k+1$,
we obtain $X_2=\left( k+p+1, (k+\frac{1}{2})\sqrt3 \right)$.
Since the second coordinate of $X_2$ is $(k+\frac{1}{2})\sqrt3$, where $k$ is integer, 
then the point $X_2$ belongs to the line, that contains the vertices $A_3$ and $A_4$. 
Since the first coordinate of $X_2$ is an integer, it follows that $X_2$ is situated in the center of the edge $A_3A_4$. 

Now we prove  that if a closed geodesic on a tetrahedron passes through the midpoint of one edge, 
then it passes through the midpoint of the opposite edge.
\begin{figure}[h]
\begin{center}
\includegraphics[width=100mm]{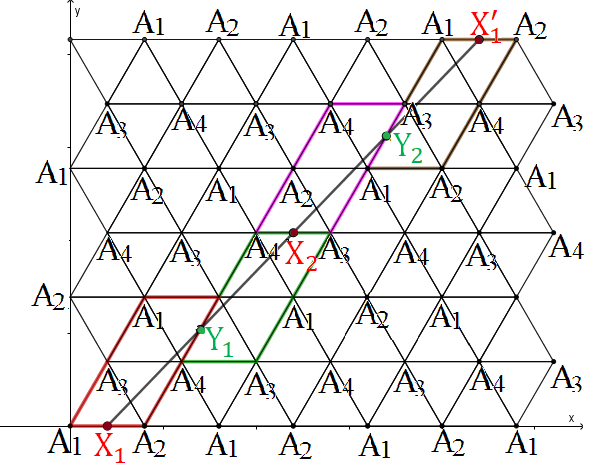}
\caption{  }
\label{vsdev2}
\end{center}
\end{figure}

Let $Y_1$ be the midpoint of $X_1X_2$. 
Then the coordinates of $Y_1$ are $(\frac{q}{4}+\frac{p}{2}+\frac{1}{2}, \frac{q}{4}\sqrt3)$.
Substituting $q=2k+1$, we obtain $Y_1 (\frac{k+p+1}{2}+\frac{1}{4}, (\frac{k}{2}+\frac{1}{4})\sqrt3)$.
From the value of the second coordinate we have that $Y_1$ belongs to the line  that
passes in the middle between the horizontal lines $y=k\frac{\sqrt3}{2}$ and $y=(k+1)\frac{\sqrt3}{2}$.
From the value of the first coordinate of $Y_1$ it follows that $Y_1$ is the center of
$A_1A_3$, or $A_3A_2$, or $A_2A_4$, or $A_4A_1$ (Figure \ref{vsdev2}).

Similarly consider the midpoint $Y_2$ of   $X_2X'_1$.
The coordinates of $Y_2$ are $(\frac{3q}{4}+\frac{3p}{2}+\frac{1}{2}, \frac{3q}{4}\sqrt3)$.
Then $Y_2$ is the midpoint of the edge that is opposite to the edge with $Y_1$.

Now we shall prove that for any closed geodesic $\gamma$ in the regular tetrahedron there exists the closed geodesic 
$\gamma_0$ equivalent to $\gamma$ and passing through the midpoint of the edge.

A geodesic equivalent to $\gamma$ is characterized by the equation
$y=\frac{q\sqrt3}{q+2p} (x-\mu) $, where $p, q$ are fixed coprime integers  
and there exist  $\mu_1, \mu_2 \in [0,1]$ such that  $\mu_1<\mu<\mu_2$.
Notice that the lines 
\begin{equation}\label{gamma1}
\gamma_1: y=\frac{q\sqrt3}{q+2p} (x-\mu_1),
\end{equation}
\begin{equation}\label{gamma2}
\gamma_2: y=\frac{q\sqrt3}{q+2p} (x-\mu_2)  
\end{equation}
pass through the vertices of the tiling.
It follows that there exist the integer numbers $c_1$ and $c_2$ such that  the points
$P_1\left( c_1\frac{q+2p}{2q}+\mu_1, c_1\frac{\sqrt{3}}{2}  \right)$
and 
$P_2\left( c_2\frac{q+2p}{2q}+\mu_2, c_2\frac{\sqrt{3}}{2} \right)$
are the vertices of the tilling and 
 $\gamma_1$ passes through  $P_1$ and $\gamma_2$ passes through $P_2$ .

Consider the closed geodesic $\gamma_0$ equivalent to $\gamma$ such that the equation of $\gamma_0$ is
\begin{equation}
y=\frac{q\sqrt3}{q+2p} \left( x-\frac{\mu_1+\mu_2}{2} \right).     \notag
\end{equation}
It is easy to proof  that the point 
$P_0$ with  coordinates 
$\left( \frac{c_1+c_2}{2} \frac{q+2p}{2q} + \frac{\mu_1+\mu_2}{2} ,  
\frac{c_1+c_2}{2} \frac{\sqrt{3}}{2}  \right)$ belongs to $\gamma_0$.
Let us show that   $P_0$ is a midpoint of some edge of the tiling.
Consider three cases.

1) Both of the points $P_1$ and $P_2$ belong to the line $A_1A_2$. \\
It follows that $c_1=2k_1$, $c_2=2k_2$ and  $x_1=l_1$, $x_2=l_2$,  where $ k_1, k_2, l_1, l_2 $ are integers,
and $P_0$ has the coordinates $\Big(  \frac{l_1+l_2}{2}, (k_1+k_2) \frac{\sqrt{3}}{2}  \Big)$.
Since  $\gamma_0$ doesn't pass through the vertices, we are reduced to two cases.
If $k_1+k_2=2k_0$, then $l_1+l_2 = 2 l_0-1$,  
and the coordinates of $P_0$ are $(l_0-\frac{1}{2}, k_0 \sqrt{3})$.
If $k_1+k_2=2k_0-1$, then $l_1+l_2 = 2 l_0$, 
and $P_0$ has the coordinates $(l_0, (2k_0 -1) \frac{\sqrt{3}}{2})$.
In both cases we get that   $P_0$ is the midpoint of some edge.

2)  Both of the points $P_1$, $P_2$  belong to the line $A_3A_4$.\\
It follows that  $c_1=2k_1+1$, $c_2=2k_2+1$  and  $x_1=l_1+\frac{1}{2}$, $x_2=l_2+\frac{1}{2}$,
  where $ k_1, k_2, l_1, l_2 $ are integers,
and  the coordinates of $P_0$ are $\left(  \frac{l_1+l_2+1}{2}, (k_1+k_2+1)  \frac{\sqrt{3}}{2} \right)$.
As above, since  $\gamma_0$ doesn't pass through the vertices, we get two cases.
If $k_1+k_2+1=2k_0-1$,  then $l_1+l_2+1= 2 l_0$, 
and $P_0$ has the coordinates $( l_0, (2k_0 -1) \frac{\sqrt{3}}{2} ) $.
If  $k_1+k_2+1=2k_0$, then $l_1+l_2+1= 2 l_0-1$ and   the coordinates of $P_0$ are $ ( l_0-\frac{1}{2}, k_0 \sqrt{3} )$.
In all this cases the point $P_0$  is the midpoint of some edge.

3) The point  $P_1$  belongs to the line $A_1A_2$ and
the point $P_2$   belongs to the line $A_3A_4$. \\
It follows that  $c_1=2k_1$, $c_2=2k_2+1$, and
 $x_1=l_1$, $x_2=l_2+\frac{1}{2}$, where $ k_1, k_2, l_1, l_2 $ are integers.
The coordinates of $P_0$ are  
$\left(  \frac{l_1+l_2}{2}+\frac{1}{4}, (k_1+k_2+\frac{1}{2}) \frac{\sqrt{3}}{2} \right) $.
We get that $P_0$  is also the center of some edge.

Finally, we found the closed geodesic  $\gamma_0$ equivalent to $\gamma$
such that $\gamma_0$  passes through the midpoint of some edge.
 As we proved before, from this it follows
 that  $\gamma_0$ passes  through the midpoints of two pairs of the opposite edges on the tetrahedron.
This completes the proof.

\end{proof}

\begin{cor}\label{corparts1}
The midpoints of two pairs of the opposite edges partition the geodesic $\gamma_0$ into four equals parts.
\end{cor}

\begin{cor}\label{corparts2}
The development of  a regular tetrahedron along a closed geodesic consists of   four equal polygons.
Any two adjacent polygons coincide modulo the point reflection with respect to the midpoint of their common edge.
\end{cor}

\begin{proof}
For any  closed geodesic $\gamma$ we get the equivalent closed geodesic $\gamma_0$
 that passes  through the midpoints of two pairs of the opposite edges on the tetrahedron.
Let the points $X_1$, $X_2$  and $Y_1$, $Y_2$ be respectively the midpoints 
of the edges $A_1A_2$, $A_4A_3$  and $A_2A_3$, $A_1A_4$.
Assume $\gamma_0$ passes through these points (Figure \ref{vsdev}).
\begin{figure}[h!]
\begin{center}
\includegraphics[width=120mm]{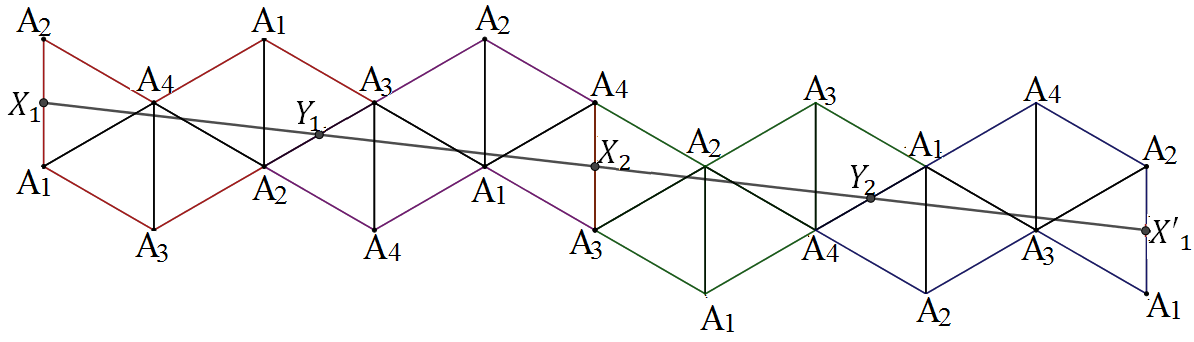}
\caption{  }
\label{vsdev}
\end{center}
\end{figure}

Consider the rotation of  the regular tetrahedron  by the angle  $\pi$  around the line passing through the points $X_1$ and $X_2$.
This rotation is the isometry of the regular tetrahedron.
The point $Y_1$ is mapped to the point $Y_2$.
Furthermore the segment of $\gamma_0$ that starts at $X_1$ on the face  $A_1A_2A_4$  
is mapped to the segment  of $\gamma_0$ that starts from the point $X_1$ on  $A_1A_2A_3$. 
It follows that the  segment $X_1Y_1$ of the geodesic is mapped to the  segment $X_1Y_2$.
For the same reason after the rotation the    segments $X_2Y_1$ and $X_2Y_2$ of  $\gamma_0$
swap.

From this rotation we get that the development of the tetrahedron along the segment $Y_1X_1Y_2$ of the geodesic is 
central-symmetric  with respect to  the point $X_1$.
And the development of the tetrahedron through the segment $Y_1X_2Y_2$ is 
central-symmetric  with respect to  $X_2$.

Now consider the rotation of the regular tetrahedron by the angle $\pi$
around the line passing through the point $Y_1$ and $Y_2$.
For the same reason we obtain that the development of the tetrahedron along the segment $X_1Y_1X_2$ of geodesic is 
central-symmetric  with respect to   $Y_1$,
and the development along the segment $X_1Y_2X_2$ is 
central-symmetric  with respect to  $Y_2$.

\end{proof}

\section{ Properties of a closed geodesics on regular tetrahedra   in   Lobachevsky space}

Let us introduce some  notaions following \cite{Pro07}.
A {\it broken line} on a tetrahedron is a curve 
that consists of the   line segments connecting points on the edges of this tetrahedron in consecutine order.
A broken line on a tetrahedron is called a {\it  generalized  geodesic} if it is closed and \\
(1) has not self-intersection points, \\
(2) passes through more than three edges on the tetrahedron and doesn't pass through its vertices, \\
(3) segments  incident to the same point on the edge belong to the adjacent faces of the tetrahedron.
 
\begin{lemma}\label{allgeod}
\textnormal{(V. Yu. Protasov~\cite{Pro07})}
For every  generalized  geodesic  on a tetrahedron  in   Euclidean space 
there exists a simple closed geodesic on a regular tetrahedron 
 in Euclidean space that is equivalent to this  generalized  geodesic.
\end{lemma}

We obtain the analog of Lemma 4.1 for hyperbolic space.
\begin{lemma}\label{middle}
A simple closed geodesic on a regular tetrahedron in   Lobachevsky space passes 
through the midpoints of two pairs of the opposite edges on the tetrahedron. 
\end{lemma}

\begin{proof}

Assume $\gamma$  is a simple closed geodesic on a regular tetrahedron $A_1A_2A_3A_4$
in   Lobachevsky space.
In the  Cayley–Klein model of  this space 
 the  regular tetrahedron   is represented by  a  regular tetrahedron in  Euclidean space as described in Section $2$.
Then the geodesic  $\gamma$  is a  generalized  geodesic on the regular tetrahedron in  Euclidean space.
From Lemma \ref{allgeod} we get that this  generalized   geodesic
is  equivalent to a closed geodesic  $\tilde\gamma$ on  the regular tetrahedron  in  Euclidean space.
 From Lemma \ref{centreqv} we assume that  $\tilde\gamma$ 
  passes through the midpoints of two pairs of the opposite edges on this tetrahedron. 
 
 Suppose  $\tilde\gamma$  passes through  the midpoints  $\tilde X_1$, $\tilde X_2$ of the edges $A_1A_2$ and $A_3A_4$.
 Let $X_1$, $X_2$  be the corresponding points on   $\gamma$.
 Consider the development $T$ of the regular tetrahedron in hyperbolic space
  along  $\gamma$ from the point $X_1$.
Then $\gamma$ is a line segment $X_1X^1_1$  on the development.

Likewise  consider the development of the  regular tetrahedron  in  Euclidean space  
along  $ \tilde \gamma$ from   $ \tilde X_1$.
From Corollary \ref{corparts2} it follows that this development is central-symmetric with respect to the point $\tilde  X_2$.

On the  regular tetrahedron in   Lobachevsky space 
denote by $M_1$, $M_2$  respectively the midpoints of the edges $A_1A_2$ and $A_3A_4$.
Consider the rotation of  the  tetrahedron  by the angle  $\pi$  around the line passing through the points  $M_1$ and $M_2$.
Since this rotation is the isometry of the  tetrahedron then the development of this  tetrahedron is  central symmetric
with respect to the point $M_2$.
Denote by $T_1$ and  $T_2$ respectively  the parts of the development along   segments $X_1X_2$ and $X_2X^1_1$.

Consider the rotation by the angle  $\pi$  around the point $M_2$.
Then the part $T_1$ coincides with $T_2$.
The edge $A_1A_2$ containing the point $X^1_1$ is mapped to the edge $A_2A_1$ with the point  $X_1$.
It follows that  the point  $X^1_1$  belongs to the edge $A_1A_2$ of the $T_1$, and the lengths of 
$A_2X_1$ and $X^1_1A_1$ are equal.

The edge $A_3A_4$ rotates to  itself but with the opposite orientation.
It follows that the point $X_2$ of the part  $T_2$   is mapped to the point $X^1_2$ of $T_1$ such that 
the  lengths of $A_4X_2$ and  $X^1_2A_3$ are equal.
Moreover,  $\angle X_1X_2A_4 = \angle X^1_1X^1_2A_4$.
Since the geodesic is closed, then $\angle A_1X_1X_2 = \angle A_1X^1_1X^1_2$  (Figure \ref{newproof}).
\begin{figure}[h]
\centering{
\includegraphics[width=90mm]{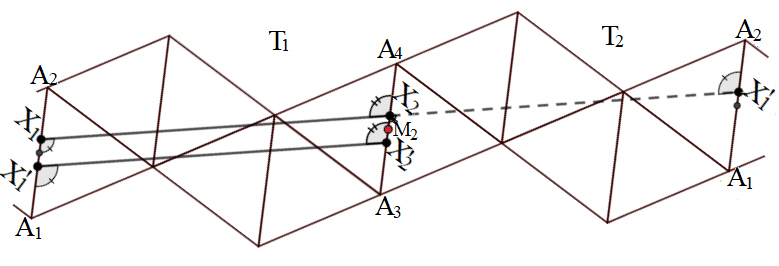} 
}
\caption{ }
\label{newproof}
\end{figure}

 We obtain the quadrilateral  $X_1X_2X^1_2X^1_1$ on $T_1$ and 
the sum of its interior angles is equal to $2\pi$.
From Gauss-Bonnet  theorem  we get that
  the   integral of  the Gaussian curvature of   Lobachevsky plane 
   over the  interior  of  $X_1X_2X^1_2X^1_1$ is equal to zero.
It follows that after the rotation the part $X^1_2X^1_1$ of the geodesic coincides with $X_1X_2$. 
Hence the points $X_1$ and $X_2$ are the midpoints of the edges (Figure \ref{newproof}).

In the same way we can proof  that   $\gamma$ passes through  the midpoints of   another two opposite edges.
This completes the proof.
\end{proof}

If $\tilde{\gamma}$ is a closed geodesic on a regular tetrahedron  in   Euclidean space,
 then there exist infinetly many closed geodesics equivalent to $\tilde{\gamma}$ \cite{FucFuc07}. 
It is not true in Lobachevsky space. 
\begin{lemma}\label{Uniqueness}
If two closed geodesic on the regular tetrahedron in   Lobachevsky space are equivalent,
then they coincide.
\end{lemma}

\begin{proof}
Let $\gamma_1$  be a   closed geodesic   on a regular tetrhedron in  Lobachevsky space.
Consider the development of this tetrhedron along $\gamma_1$ from a point $X$ of the edge $A_1A_2$. 
Then $\gamma_1$  corresponds to the  segment $XX'$ on the development and  $\angle XX'A_1 = \angle X'XA_2$ 
 (Figure \ref{uniq}).

Suppose $\gamma_2$ is a closed geodesic equivalent  to $\gamma_1$ and $\gamma_2$ does not coinside with $\gamma_1$.
Then $\gamma_2$ starts at the point $Y$ on the edge $A_1A_2$ and $Y$ differs from $X$.
Since $\gamma_2$ and $\gamma_1$  pass through the same edges in the same order, 
then $\gamma_2$ corresponds to the   segment $YY'$ on the same development and  $\angle YY'A_1 = \angle Y'YA_2$.
\begin{figure}[h]
\begin{center}
\begin{minipage}[h]{0.4\linewidth}
\includegraphics[width=75mm]{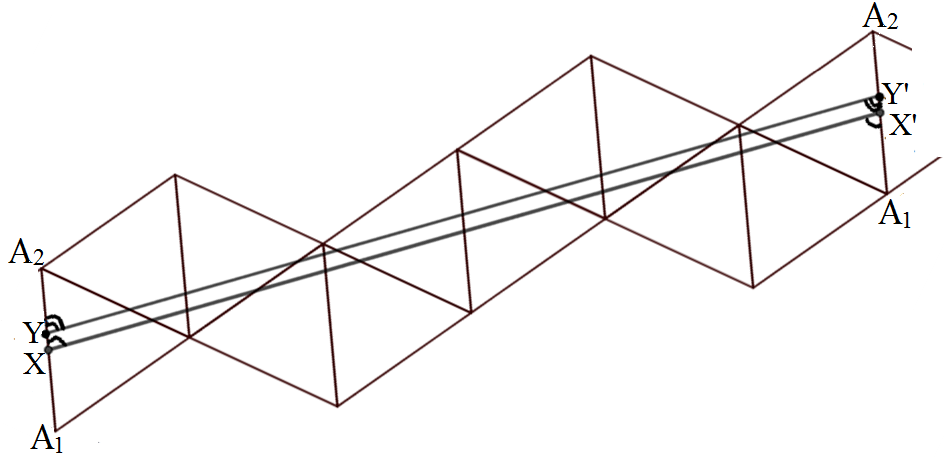} 
\caption{ }
\label{uniq}
\end{minipage}
\hfill 
\begin{minipage}[h]{0.5\linewidth}
\includegraphics[width=70mm]{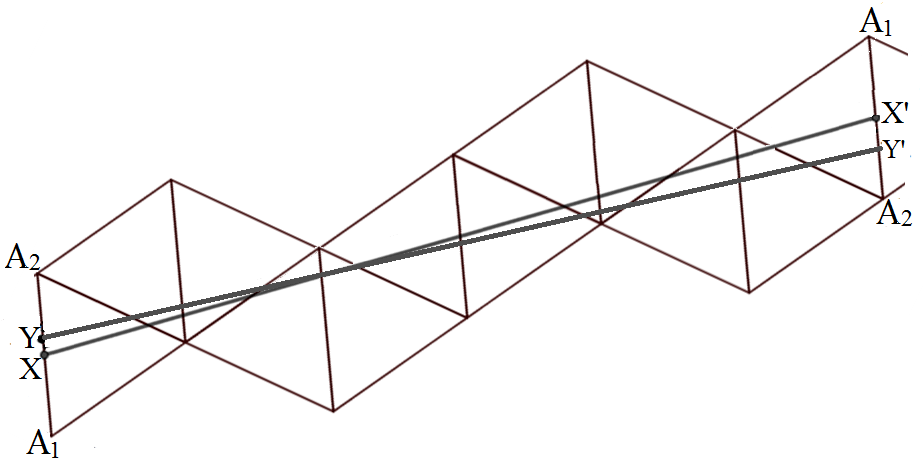}
\caption{  }
\label{mobius}
\end{minipage}
\end{center}
\end{figure}

Assume that    $YY'$ intersects $XX'$ on the development.
It follows that the edge with the points $X$ and $Y$ and the edge with the points $X'$ and $Y'$ have the opposite orientation.
Since the edges with this points correspond to the edge $A_1A_2$ on the tetrahedron,
 we obtain that  the tetrahedron contains a Mobius strip  (Figure \ref{mobius}).
 It contradicts the orientation of the tetrahedron.
 
 Thus the segments $YY'$ and $XX'$ do not intersect.
 We get a  quadrilateral $XX'YY'$ (Figure \ref{uniq}).
 The sum of its interior angles is equal to $2\pi$.
As before from Gauss-Bonnet theorem   it follows that  $XX'$ and $YY'$  are coincident.
It follows that the closed geodesic $\gamma_2$ coincides with the closed geodesic  $\gamma_1$.

\end{proof}

\section{Proof of Theorem 1}\label{nmstructure}

\begin{lemma}\label{pq}\textnormal{(V. Yu. Protasov~\cite{Pro07})}A simple closed geodesic $\gamma$ on a regular tetrahedron in Euclidean space is 
uniquely characterized by the two coprime integers  $(p,q)$   such that 
there are $p$ points of $\gamma$ on   each of two opposite edges of the tetrahedron,
$q$ points of $\gamma$ on   each of another two opposite edges,
and there are $(p+q)$ points of $\gamma$ on    each edges of the third pair of opposite edges.
\end{lemma}

\begin{remark}
By  two coprime integers  $(p,q)$  we mean two integer numbers $p$, $q$ such that $0<p<q$ and GCD$(p,q)=1$, 
or $p=0$, $q=1$.
\end{remark}

This pair of coprime integers  $(p,q)$   is called a {\it type} of a simple closed geodesic.
Moreover, for any two coprime integers $(p,q)$ there exist infinitely many simple closed geodesics of type $(p,q)$
on a regular tetrahedron in Euclidean space and all of these geodesics are equivalent to each other.
From Lemma \ref{centreqv} it follows that in the set of geodesics of type $(p,q)$
there exists a geodesic $\tilde{\gamma}$ 
passing through the midpoints of two pairs of the opposite edges on the tetrahedron.
Denote by $\tilde X_1$ and $\tilde X_2$ the midpoints of the edges $A_1A_2$ and $A_3A_4$
and by $\tilde Y_1$, $\tilde Y_2$ the midpoints of $A_1A_3$ and $A_2A_4$.
Assume that $\tilde{\gamma}$ passes through $\tilde X_1$, $\tilde X_2$ and $\tilde Y_1$, $\tilde Y_2$.
Consider the development of the tetrahedron along $\tilde{\gamma}$ from the point $\tilde X_1$ 
through the point $\tilde Y_1$ to the point $\tilde X_2$.
It is the polygon $\tilde{T}_1$ on Euclidean plane.
Furthermore, $\tilde{T}_1$ is a central-symmetric with respect to $\tilde Y_1$. 
The interior angles  of $\tilde{T}_1$ are equal to $\frac{\pi}{3}$, or $\frac{2\pi}{3}$, or $\pi$, or $\frac{4\pi}{3}$.
Note that the angle is equal to $\frac{4\pi}{3}$ if and only if 
$\tilde{\gamma}$ intersects successively three edges sharing a common vertex.

Now consider the set of congruent regular triangles with the plane  angle  equal to  $\alpha$ on Lobachevsky plane.
Put these triangles in the same order as we develop the faces of the regular tetrahedron in Euclidean space along 
the  part $\tilde X_1 \tilde Y_1 \tilde X_2$ of  $\tilde{\gamma}$.
In other words, we build a polygon $T_1$ on Lobachevsky plane that consists of the triangles 
ordered in the same way as $\tilde{T}_1$ on Euclidean plane.
Label the vertices of $T_1$ according to the vertices of $\tilde{T}_1$.
Then the polygon $T_1$ corresponds to some development of a regular tetrahedron with 
the plane angles of the faces equal to $\alpha$ in Lobachevsky space.
Moreover, $T_1$ is also central-symmetric with respect to the midpoint of the same edge $A_1A_3$ as the polygon $\tilde{T}_1$.
By construction, the interior angles at the vertices of $T_1$ are equal to
$\alpha$ or $2\alpha$, or $3\alpha$, or $4\alpha$.

Consider the midpoints $ X_1$ and $X_2$ of the edges $A_1A_2$ and $A_3A_4$ on $T_1$
that $ X_1$ and $X_2$ correspond to the points $\tilde X_1$ and $\tilde X_2$ on $\tilde{T}_1$.
Construct the line segment $ X_1X_2$.

Assume $\alpha \in (0, \frac{\pi}{4}]$.
Hence the polygon $T_1$ is convex.
It follows that the segment $ X_1X_2$ belongs to the interior of the polygon $T_1$. 
Moreover, $ X_1X_2$ passes through the center of symmetry of $T_1$.

In the same way consider the development $\tilde T_2$ of the tetrahedron along the second part $\tilde X_2 \tilde Y_2 \tilde X'_1$
of the geodesic $\tilde{\gamma}$ in Euclidean space.
Note that $\tilde T_1$ and $\tilde T_2$   are equal   polygons
(without vertex labeling).
Then consider the second copy  of   $T_1$  on hyperbolic plane, denote it by $T_2$ and label   vertices of  $T_2$ 
according to the vertices on $\tilde T_2$.
Mark by $ X_2$ and $X'_1$ the midpoints of $A_3A_4$ and $A_1A_2$ on $T_2$
so that $ X_2$ and $X'_1$  correspond to the points $\tilde X_2$ and $\tilde X'_1$.
Since $\alpha \in (0, \frac{\pi}{4}]$ then the segment $X_2 X'_1$ also
belongs to the interior of $T$ and passes through its center of symmetry.

These two polygons $T_1$ and $ T_2$ correspond  to the development of a regular tetrahedron with 
the plane angles of the faces equal to $\alpha$ in Lobachevsky space.
Then   $ X_1X_2$ and $ X_2X'_1$ are the geodesic segments on this tetrahedron.
Furthermore,  on the tetrahedron the point $X_2$ of $X_1X_2$ coincide with $X_2$ of $X_2X'_1$ 
and the point $X_1$ of the segment $X_1X_2$ coincide with $X'_1$ of $X_2X'_1$.
Since $T_1$ and $ T_2$   is central symmetric then $\angle X_1X_2A_3 = \angle X'_1X_2A_4$ 
and $\angle X_2X_1A_1 = \angle X_2X'_1A_2$.
The segments $X_1X_2$ and $ X_2X'_1$ are equivalent  to the geodesic segments
$\tilde X_1 \tilde X_2$ and $\tilde X_2 \tilde X'_1$ on the regular tetrahedron in Euclidian space respectively.
Therefore,   $ X_1X_2$ and $ X_2X'_1$ form a simple closed geodesic $\gamma$ 
on the regular tetrahedron with the plane angles equal to $\alpha \in (0, \frac{\pi}{4}]$ in Lobachevsky space.

Increase the angle $\alpha$ from the $ \frac{\pi}{4}$.
Then the polygon $T$ is not convex because it contains the interior angles   equal  to $4\alpha>\pi$.

Let $\alpha_0$ be the supremum of 
 the set of $\alpha$ such that the segment $ X_1X_2$ belongs to the interior of $T_1$.
Assume $\alpha_0< \frac{\pi}{3}$.
For all $\alpha<\alpha_0$ the segment $ X_1X_2$ is a part of a simple closed geodesic  $\gamma$  on 
the regular tetrahedron in Lobachevsky space.
The distance   from  the set of vertices of the  tetrahedron to $\gamma$
satisfies inequality (\ref{distvertex2}).
It follows that there exists 
sufficiently small $\varepsilon$ such that for
$\alpha_1=\alpha_0+\varepsilon$   the segment $ X_1X_2$ belongs to the interior of $T$.
It contradicts to  the supremum of this set.
Then $\alpha_0= \frac{\pi}{3}$.

Finally, we obtain that for all $\alpha \in (0, \frac{\pi}{3})$ the segments $ X_1X_2$ and $X_2 X'_1$ belong to the interior of $T$.
From this it follows

\begin{theorem*}
On a regular tetrahedron in Lobachevsky space for any coprime integers $(p, q)$, $0\le p<q$, 
there exists  unique, up to the rigid motion of the tetrahedron, simple closed geodesic  of type $(p,q)$.
The geodesics of type $(p,q)$ exhaust all simple closed geodesics on a regular tetrahedron in  Lobachevsky space.
\end{theorem*}

From Lemma \ref{Uniqueness} it follows the uniqueness of a simple closed geodesic of type $(p,q)$
on a regular tetrahedron in Lobachevsky space.
Such geodesic has $p$ points on  each of two opposite edges of the tetrahedron,
$q$ points on  each of another two opposite edges,
and $(p+q)$ points on   each edges of the third pair of opposite edges.
Hence for any coprime integers $(p, q)$, $0 \le p<q$, there exist three simple closed geodesic of type $(p,q)$
on a regular tetrahedron in Lobachevsky space.
They coincide by the rotation of the tetrahedron
by the angle $\frac{2\pi}{3}$ or $\frac{4\pi}{3}$ about the altitude constructed from a vertex to the opposite face.

Since any simple closed geodesic on a regular tetrahedron in Lobachevsky space is equivalent to 
a simple closed geodesic on a regular tetrahedron in Euclidean space, then there are not another 
simple closed geodesic on a regular tetrahedron in Lobachevsky space.

\begin{figure}[h]
\begin{center}
\begin{minipage}[h]{0.3\linewidth}
\includegraphics[width=50mm]{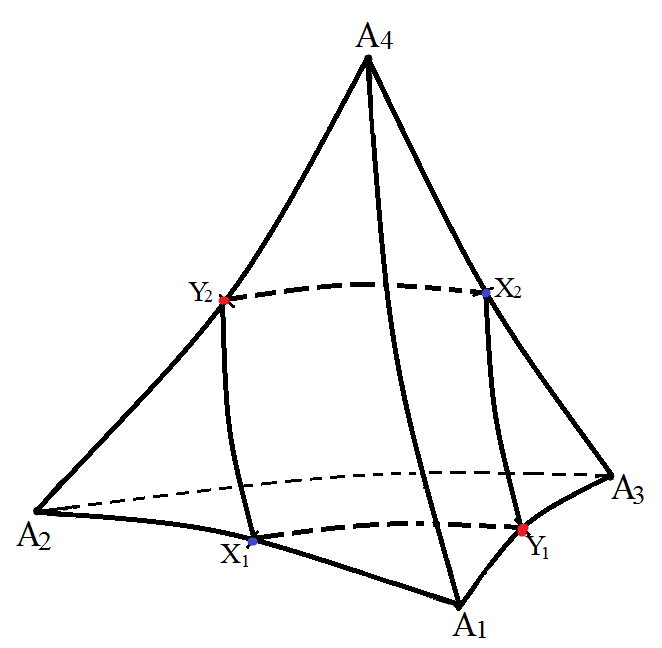} 
\end{minipage}
\hfill 
\begin{minipage}[h]{0.3\linewidth}
\includegraphics[width=50mm]{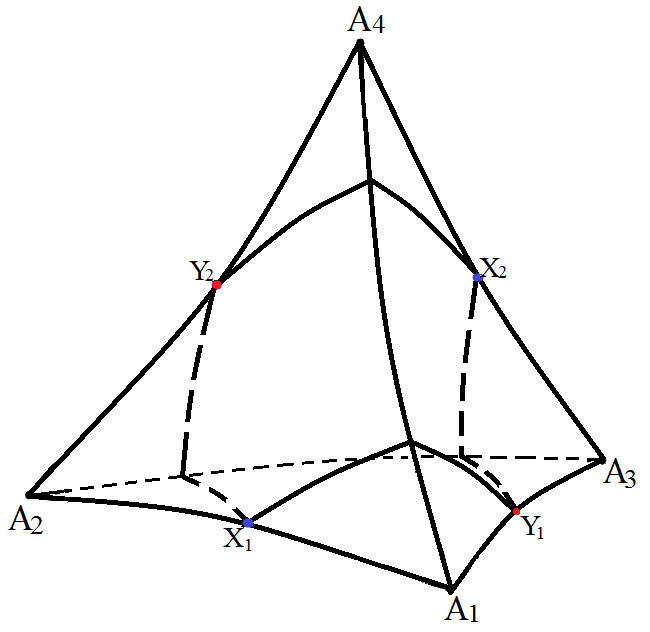}
 \end{minipage}
\hfill 
\begin{minipage}[h]{0.3\linewidth}
\includegraphics[width=50mm]{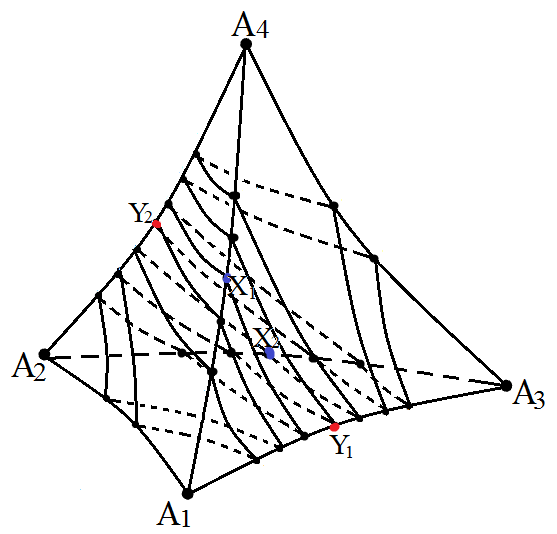}
\end{minipage}
\end{center}
 \caption{A simple closed geodesics of type $(0,1)$, $(1,1)$ and $(5,2)$
  respectively on a regular tetrahedron in Lobachevsky space.  }
\label{geodtetr}
\end{figure}

\section{Proof of Theorem 2}

By vertex of a geodesic we understand a point of this geodesic that belongs to an edge of a tetrahedron.
Consider  vertices  $B_0$, $B^1_1$, $B^2_1$   of the geodesic such that the segments  $B_0 B^1_1$ and $B_0 B^2_1$
are adjacent segments of the geodesic.
The vertex $B_0$     is called a {\it  catching point}    if    $B_0$, $B^1_1$ and $B^2_1$
    lie on the three edges  sharing a common vertex  on the tetrahedron  
  and  $B_0$, $B^1_1$, $B^2_1$ are the nearest vertices of geodesic to this vertex of the tetrahedron.

In \cite{Pro07}  V. Yu. Protasov shows following property of a simple closed geodesic    on a regular tetrahedron 
in Euclidean space.
\begin{lemma}\label{gengeodstr1}\textnormal{\cite{Pro07}}
Let $\gamma^1_1$, $\gamma^2_1$ be the segments of a simple closed geodesic $\gamma$ comming out from a  catching  point
on a regular tetrahedron in Euclidean space.
The segments   $\gamma^2_2$, $\gamma^2_2$   of  $\gamma$ are the segments follow  $\gamma^1_1$, $\gamma^2_1$
respectively and so on.
Then for each $k=2, \dots, 2p+2q-1$  the segments $\gamma^1_k$, $\gamma^2_k$ belong to the same face of the tetahedron
 and   there are no   points of $\gamma$ between $\gamma^1_k$ and $\gamma^2_k$.
The last segments $\gamma^1_{2p+2q}$, $\gamma^2_{2p+2q}$ share the second  catching point.
\end{lemma}

Since any simple closed geodesic $\gamma$ on a regular tetrahedron in Lobachevsky space is equivalent to 
a simple closed geodesic on a regular tetrahedron in Euclidean space, then $\gamma$ also satisfy Lemma \ref{gengeodstr1}.

\begin{lemma}\label{condpq}
If the length of a  simple closed geodesic  of type $(p,q)$ on a regular tetrahedron in Lobachevsky space is not greater than $L$, 
then 
\begin{eqnarray}\label{pqcond}
&p+q \le \frac{3}{4} 
\frac {  L -  2 \ln  \left(    \frac{2\pi^3 - \left( \pi - 3\alpha \right)^3  \left(1- \frac{4\alpha^2}{\pi^2}\right) } 
{2\pi^3 - \left( \pi - 3\alpha \right)^3  \left(1+ \frac{4\alpha^2}{\pi^2}\right) } \right)  } 
{    \ln \left(   \frac{2\pi^3 - \left( \pi - 3\alpha \right)^3 \left(1- \frac{\alpha^2}{\pi^2}\right) }
{2\pi^3 - \left( \pi - 3\alpha \right)^3\left(1+ \frac{\alpha^2}{\pi^2}\right) } \right) +
\ln \left(  \frac{3  \pi-3\alpha}{\pi + 3\alpha}   \right)   } 
+2&
 \end{eqnarray}
where $\alpha$ is the plane angle  of a face of the tetrahedron.
\end{lemma}

\begin{proof}
1. Consider a regular tetrahedron $A_1A_2A_3A_4$ with the plane angle  of a face  equal  to $\alpha$ in Lobachevsky space.
Let $\gamma $ be a simple closed geodesic of type $(p,q)$ on this tetrahedron.
Assume that the vertices $B_1$, $B_2$, $B_3$ of $\gamma $  belong to the edges 
$A_4A_1$, $A_4A_2$, $A_4A_3$ respectively and
these vertices are the nearest to the vertex $A_4$.
Thus the segments $B_2B_1$ and  $B_2B_3$   corespond to the segments $\gamma^1_1$ and $\gamma^2_1$ 
sharing a catching point $B_2$.

Consider the development of the   faces $A_1A_4A_2$ and $A_2A_4A_3$ on hyperbolic plane.
The segments $\gamma^1_1$ and $\gamma^2_1$  correspond to the one segment $B_1B_3$ on this development.
Without loss of generality we assume that the length of $A_4B_1$  is not less than the length of  $A_4B_3$.

Let us check that the  function of the distance  between $A_4$ and  points on $\gamma$
attains its minimum at the point $H$ belonging to the segment  $B_1 B_3$.
Suppose it is not true. 
We will extend the geodesic $\gamma$ from the point $B_3$ on the face $A_1A_4A_3$.
From Corollary \ref{convexdistcor} we obtain that the distance between $A_4$ and  $\gamma$ is decrease.
Therefore,  $\gamma$  intersects the edge $A_1A_4$ on the point belonging to the segment $A_4B_1$.
But the vertex $B_1$ is the nearest  to   $A_4$  vertex of $\gamma$  on the edge  $A_1A_4$. 
We get a contradiction.
Then   the point $ H$
belonging to the  segment $B_1B_3$ and $A_4H$ is perpendicular to $B_1B_3$.

\begin{figure}[h!]
\begin{center}
\includegraphics[width=90mm]{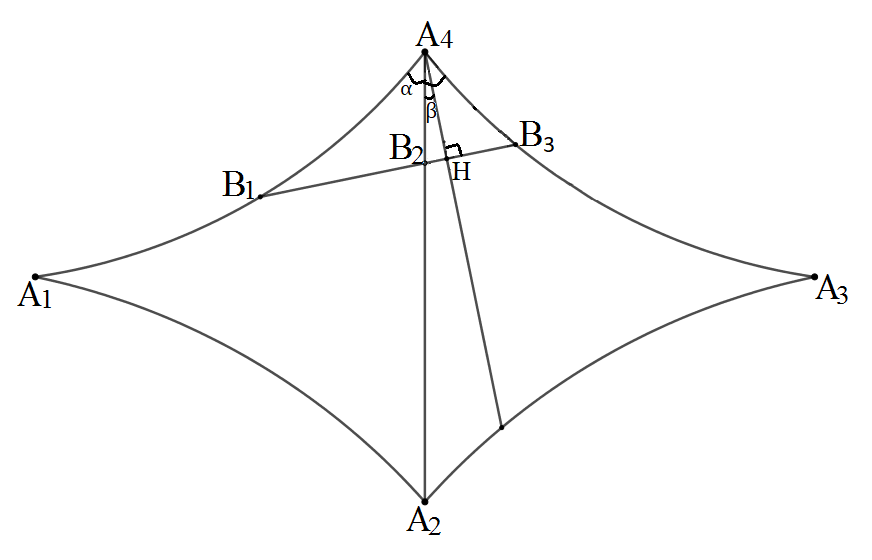}
\caption{ }
\label{abdev}
\end{center}
\end{figure}

Denote the angle $\angle B_2A_4H$ by $\beta$  and $ \beta \in [-\alpha, \alpha]$ (Figure \ref{abdev}).
Consider the length of $B_1B_3$ as a fuction of $\beta$.
From equation (\ref{thshtan}) applied to the   triangles $B_1A_4H$ and $B_3A_4H$ we obtain 
\begin{equation}\label{B1H}
\tanh |B_1H| =  \sinh |A_4H| \tan (\alpha+\beta).
\end{equation}
\begin{equation}\label{B3H}
\tanh |B_3H| = \sinh |A_4H| \tan (\alpha-\beta).
\end{equation}
From equations  (\ref{B1H}) and (\ref{B3H}) it follows that 
\begin{equation}\label{BB2s}
|B_1B_3|=\textnormal{artanh} \big( \sinh |A_4H| \tan (\alpha+\beta) \big) + \textnormal{artanh} \big( \sinh |A_4H| \tan (\alpha-\beta) \big).
\end{equation}
Find the first derivative of the function in the right side of  (\ref{BB2s}),
set the derivative equal to zero and solve for $\beta$.
We obtain that the  length of $B_1B_3$  attains its minimum at $\beta = 0$ and this minimum is  
\begin{equation}\label{BB2s1}
|B_1B_3|_{min}=2\textnormal{artanh} \big( \sinh |A_4H| \tan \alpha \big). 
\end{equation}

From Lemma  \ref{nesdist} we get that the distance between   $A_4$ and  $\gamma$  satisfies follow  inequality 
\begin{equation}
\tanh |A_4H|>\cos\frac{3\alpha}{2}\cos\frac{\alpha}{2} \frac{\sqrt{2\cos\alpha-1}}{\cos\alpha}.  \notag
 \end{equation}

Using $\sinh |A _4H| = \frac{\tanh |A_4H|}{\sqrt{ 1-\tanh^2|A_4H| } }$, we obtain 
\begin{equation}\label{A4H}
\sinh |A_4H| > \frac{  \cos \alpha \cos\frac{3\alpha}{2} \sqrt{  \cos^3 \frac{3\alpha}{2} \cos\frac{\alpha}{2}  }  }
{ \cos^2\alpha - \cos^3 \frac{3\alpha}{2} \cos\frac{\alpha}{2}  }. 
 \end{equation}

Combining (\ref{A4H}) and  (\ref{BB2s1}), we get follow estimation of the length  $B_1B_3$  
\begin{equation}\label{B1B3}
\tanh \frac{|B_1B_3|}{2} >  \frac{  \sin \alpha \cos\frac{3\alpha}{2} \sqrt{  \cos^3 \frac{3\alpha}{2} \cos\frac{\alpha}{2}  }  }
{ \cos^2\alpha - \cos^3 \frac{3\alpha}{2} \cos\frac{\alpha}{2}  }. 
 \end{equation}

The length of the segments  $\gamma^1_{2p+2q}$ and $\gamma^2_{2p+2q}$ comming out from
 the second catching point also satisfies 
inequality (\ref{B1B3}).

2. Consider the development of  the tetrahedron   along the segments $\gamma^1_i$,
$i=2, \dots, 2p+2q-1$ of the geodesic in   hyperbolic space.
These segments   form a line segment $\gamma^1$ inside  the development.
From Lemma \ref{gengeodstr1} we have that 
the segments $\gamma^2_i$, $i=2, \dots, 2p+2q-1$  
intersect the same sequence of the edges.
It follows that   $\gamma^2_i$,  
$i=2, \dots, 2p+2q-1$ form  the segment $\gamma^2$ inside  the development and 
 $\gamma^2$ does not intersect $\gamma^1$.

The development along $\gamma^1$ consist of $(2q+2p-2)$  faces of the tetrahedron.
Consider three  consecutive faces of the development containig the segments $\gamma^1_{i-1}$, $\gamma^1_i$, 
$\gamma^1_{i+1}$, $i=3 \dots 2p+2q-2$.
Since $\gamma$  is a simple closed geodesic,  it follows that $\gamma^1$ does not pass throught four edges 
 sharing a common vertex of  the tetrahedron.
Therefore, the segments $\gamma^1_{i-1}$, $\gamma^1_i$, $\gamma^1_{i+1}$ intersect
 two opposite edges of the tetrahedron.
Assume that $\gamma^1_{i-1}$, $\gamma^1_i$ and $\gamma^1_{i+1}$
  intersect  the edges $A_4A_3$, $A_4A_1$, $A_4A_2$ and $A_2A_3$ at the points 
  $B_{i-2}$, $B_{i-1}$,  $B_i$ and $B_{i+1}$ respectively  (Figure \ref{condminsmall}).

\begin{figure}[h]
\begin{center}
\begin{minipage}[h]{0.4\linewidth}
\includegraphics[width=80mm]{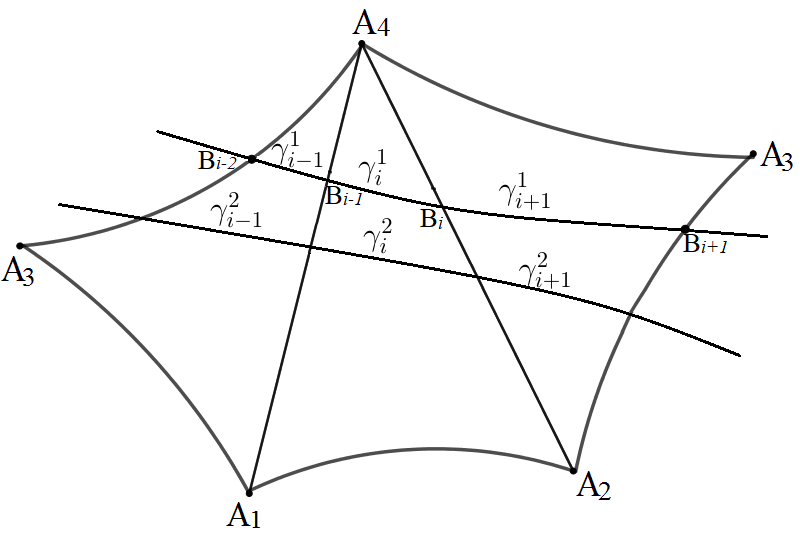}
\caption{ }
\label{condminsmall}
\end{minipage}
\hfill 
\begin{minipage}[h]{0.5\linewidth}
\includegraphics[width=80mm]{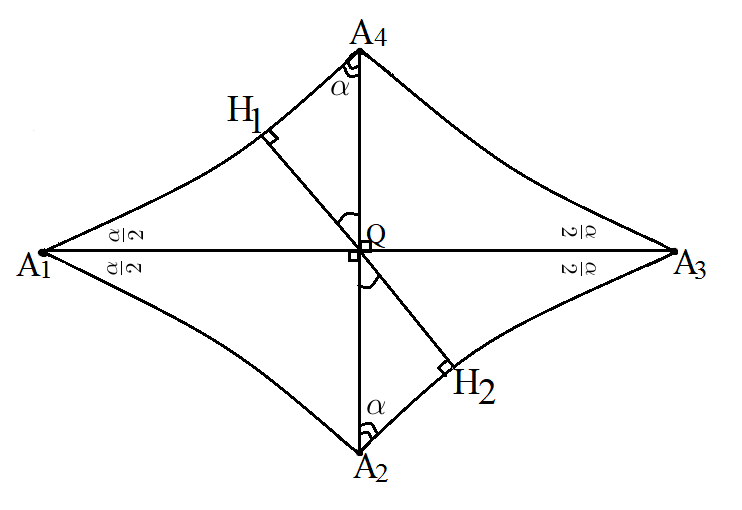}
\caption{}
\label{rhomimax}
\end{minipage}
\end{center}
\end{figure}

Similarly to the estimation of the length  $B_1B_3$,  we obtain following inequality for  the length of $B_{i-2} B_{i-1}$.
 \begin{equation}\label{Bi2Bi1}
\tanh \frac{|B_{i-2}B_{i-1}|}{2} > 
 \frac{ \tan  \frac{\alpha}{2} \cos \alpha \cos\frac{3\alpha}{2} \sqrt{  \cos^3 \frac{3\alpha}{2} \cos\frac{\alpha}{2}  }  }
{ \cos^2\alpha - \cos^3 \frac{3\alpha}{2} \cos\frac{\alpha}{2}  }. 
 \end{equation}

For evaluating the length of the segment $B_{i-1}B_{i+1}$ consider the development of two adjacent faces
$A_4A_2A_1$ and $A_4A_2A_3$ containing this segment.
Denote by  $Q$ the midpoint of the edge $A_2A_4$.
On the edges  $A_4A_1$ and $A_2A_3$ mark the points $H_1$ and $H_2$  respectively so that 
$QH_1$ is perpendicular to $A_4A_1$ and $QH_2$ is perpendicular to $A_2A_3$.
Since $\angle H_1A_4Q = \angle H_2A_2Q = \alpha < \frac{\pi}{2}$ and the sum of the interior angles of the hyperbolic triangle 
is less then $\pi$, it follows that the segments $QH_1$, $QH_2$ 
belong to the triangles $A_1QA_4$, $A_3QA_2$ respectively.

The segment $A_1A_3$  is  perpendicular to $A_4A_2$ at the point  $Q$ (Figure \ref{rhomimax}).
The lengths of $A_1Q$ and $QA_3$ are equal.
It follows that the faces $A_1A_2A_4$, $A_2A_4A_3$ 
 coincide modulo the rotation by the angle $\pi$ about $Q$.
After this rotation the point $H_2$ on $A_2A_3$  is mapped to the point $H_1$ on $A_4A_1$.
Then  $\angle A_4QH_1 = \angle A_2QH_2$.
We obtain that  $H_1QH_2$ is a straight line segment and the
function of a distance  between  $A_1A_4$ and $A_2A_3$
attains its minimum at $H_1 H_2$.
From  equation (\ref{shshsin}) applied to  the triangle $QH_1A_4$ we obtain
\begin{equation}\label{QH11}
\sinh |QH_1| = \sinh |QA_4| \sin \alpha = \sinh \frac{a}{2} \sin \alpha.
\end{equation}
From equation (\ref{sha2}) we get
\begin{equation}\label{shQH1}
\sinh |QH_1| = \frac{\sqrt{2\cos \alpha-1}}{2\sin \frac{\alpha}{2}} \sin \alpha = \cos \frac{\alpha}{2} \sqrt{2\cos \alpha-1}. 
\end{equation}

Let us check that the 
distance from  the set of  the vertices of the tetrahedron to  $H_1H_2$ satisfies 
necessary condition  (\ref{distvertex}).
From  the hyperbolic  Pythagorean theorem  (\ref{Pifagor}) applied to  the triangle $A_4H_1Q$  we get
\begin{equation}
\cosh|A_4H_1|= \frac{\cosh |A_4Q|}{\cosh |QH_1|} = \frac{\cosh \frac{a}{2}}{\sqrt{1+\sinh^2 |QH_1|}}. 
\notag
\end{equation}
Using  (\ref{cha2}) and (\ref{shQH1}), we obtain
\begin{equation}\label{s01}
\cosh|A_4H_1|= \frac{1}{2\sin \frac{\alpha}{2} \sqrt{\cos^2\alpha+ \cos^2 \frac{\alpha}{2} } }.
\end{equation}
Transforming (\ref{s01}), we get
\begin{equation}\label{s0}
\tanh|A_4H_1|= 1 - 4\sin^2 \frac{\alpha}{2} \left( \cos^2\alpha+ \cos^2 \frac{\alpha}{2} \right).
\end{equation}
From $\tanh|A_4H_1| >\cos \frac{3\alpha}{2} \cos \frac{\alpha}{2} \frac{\sqrt{2\cos\alpha-1}}{\cos\alpha}$ it follows that
the necessary condition  (\ref{distvertex}) is satisfied.

Therefore, from (\ref{shQH1}) we get 
\begin{equation}\label{Bi1Bi1}
\sinh\frac{|B_{i-1}B_{i+1}|}{2} \ge \cos \frac{\alpha}{2} \sqrt{2\cos \alpha-1}. \notag
\end{equation}
Transforming the last inequality, we obtain
\begin{equation}\label{Bi1Bi1}
\tanh\frac{|B_{i-1}B_{i+1}|}{2} \ge \frac{\cos \frac{\alpha}{2} \sqrt{2\cos \alpha-1} }
{ \sqrt{\cos^2\alpha+ \cos^2 \frac{\alpha}{2} } }. 
\end{equation}

Since  the length of a  simple closed geodesic $\gamma$  of type $(p,q)$ 
on a regular tetrahedron in Lobachevsky space is not greater than $L$, we get 
\begin{eqnarray}\label{LBB}
&L \ge  2\left(  \frac{2p+2q-4}{3}  \right) 
\Big(  |B_{i-2}B_{i-1}| +|B_{i-1}B_{i+1}|     \Big) + 2 |B_1B_3|. &
 \end{eqnarray}
 Using (\ref{B1B3}), (\ref{Bi2Bi1}), (\ref{Bi1Bi1}), we obtain
 \begin{eqnarray*}
&L \ge  8 \left(  \frac{p+q-2}{3}  \right) 
\left( 
\textnormal{artanh} \left(    
\frac{ \tan \frac{\alpha}{2} \cos \alpha \cos\frac{3\alpha}{2} \sqrt{  \cos^3 \frac{3\alpha}{2} \cos\frac{\alpha}{2}  }  }
{ \cos^2\alpha - \cos^3 \frac{3\alpha}{2} \cos\frac{\alpha}{2}  }   
 \right)  +
\textnormal{artanh} \left(  
\frac{\cos \frac{\alpha}{2} \sqrt{2\cos \alpha-1} }{ \sqrt{\cos^2\alpha+ \cos^2 \frac{\alpha}{2} }     }
 \right)    
\right)  &\\  \notag
&+ 4 \textnormal{artanh} \left(  
\frac{  \sin \alpha \cos\frac{3\alpha}{2} \sqrt{  \cos^3 \frac{3\alpha}{2} \cos\frac{\alpha}{2}  }  }
{ \cos^2\alpha - \cos^3 \frac{3\alpha}{2} \cos\frac{\alpha}{2}  }   
 \right). &  
 \end{eqnarray*}

Combining inequalities  (\ref{siny}), (\ref{cosy}) and   $\cos \frac{\alpha}{2} > \frac{\sqrt{3}}{2}$ for $0<\alpha<\frac{\pi}{3}$
with (\ref{B1B3}), (\ref{Bi2Bi1}) and (\ref{Bi1Bi1}), we obtain 
 \begin{equation}
\tanh \frac{|B_1B_3|}{2} \ge  \frac{2 \alpha}{\pi} 
\sqrt{   \frac{\left( \pi - 3\alpha \right)^3}{2\pi^3 - \left( \pi - 3\alpha \right)^3}   };  \notag
 \end{equation}
 \begin{equation}
\tanh \frac{|B_{i-2}B_{i-1}|}{2}  \ge  \frac{ \alpha}{\pi} 
\sqrt{   \frac{\left( \pi - 3\alpha \right)^3}{2\pi^3 - \left( \pi - 3\alpha \right)^3}   };   \notag
 \end{equation}
 \begin{equation}
\tanh\frac{|B_{i-1}B_{i+1}|}{2} \ge \frac{\sqrt{  \pi-3\alpha} }{4\sqrt{ \pi}}.    \notag
\end{equation}
Transforming these  equations, we obtain
 \begin{equation}\label{B1B3n}
|B_1B_3| \ge \ln
 \left( 
  \frac{2\pi^3 - \left( \pi - 3\alpha \right)^3  \left(1- \frac{4\alpha^2}{\pi^2}\right) } 
{2\pi^3 - \left( \pi - 3\alpha \right)^3  \left(1+ \frac{4\alpha^2}{\pi^2}\right) } 
\right); 
 \end{equation}
 \begin{equation}\label{Bi2Bi1n}
|B_{i-2}B_{i-1}| \ge  \ln \left(   \frac{2\pi^3 - \left( \pi - 3\alpha \right)^3 \left(1- \frac{\alpha^2}{\pi^2}\right) }
{2\pi^3 - \left( \pi - 3\alpha \right)^3\left(1+ \frac{\alpha^2}{\pi^2}\right) } \right); 
 \end{equation}
 \begin{equation}\label{Bi1Bi1n}
|B_{i-1}B_{i+1}| \ge \ln \left(  \frac{3  \pi-3\alpha}{\pi + 3\alpha}   \right).
\end{equation}

Putting  (\ref{B1B3n}), (\ref{Bi2Bi1n}), (\ref{Bi1Bi1n})  in  (\ref{LBB}), we have
\begin{eqnarray}\label{Lvs1}
&L \ge  4 \left(  \frac{p+q-2}{3}  \right) 
\Big(
  \ln \left(   \frac{2\pi^3 - \left( \pi - 3\alpha \right)^3 \left(1- \frac{\alpha^2}{\pi^2}\right) }
{2\pi^3 - \left( \pi - 3\alpha \right)^3\left(1+ \frac{\alpha^2}{\pi^2}\right) } \right) +
\ln \left(  \frac{3  \pi-3\alpha}{\pi + 3\alpha}   \right)
   \Big) 
+ 2 \ln
 \left( 
  \frac{2\pi^3 - \left( \pi - 3\alpha \right)^3  \left(1- \frac{4\alpha^2}{\pi^2}\right) } 
{2\pi^3 - \left( \pi - 3\alpha \right)^3  \left(1+ \frac{4\alpha^2}{\pi^2}\right) } 
\right). &
 \end{eqnarray}

From (\ref{Lvs1}) it follows required estimation (\ref{pqcond}).
This complete the proof.

\end{proof}

Euler's  function $\phi(n)$ is equal to the number of
positive integers not greater than  $n $ and prime to $n \in \mathbb N$.
From \cite[Th. 330.]{Hard} we know
\begin{equation}\label{phi}
 \sum\limits_{n=1}^{x}  \phi(n) = \frac{3}{\pi^2}x^2+O(x\ln x),
\end{equation}
where $O(x\ln x) < C x\ln x $, when $x \rightarrow +\infty$.

Denote by $\psi(x)$ a number of  pair of coprime integers $(p,q)$ such that $p<q$ and $p+q\le x,$ $x\in \mathbb R$.
\begin{lemma}\label{psilem}
The asymptotic behavior of $\psi(x)$ is
\begin{equation}\label{psihatpsi2}
\psi(x) = \frac{3}{2\pi^2}x^2+O(x\ln x),
\end{equation}
where $O(x\ln x) < C x\ln x $ when $x \rightarrow +\infty$.
\end{lemma}

\begin{proof}

Suppose  $\hat \psi(y)$  is equal to the number of  pair of coprime integers $(p,q)$ such that $p<q$ and 
$p+q= y$, $y\in \mathbb N$.
From the definitions we get
\begin{equation}\label{psihatpsi}
\psi(x) = \sum\limits_{y=1}^{x} \hat \psi(y) 
\end{equation}

It is easy to proof that if $(p, q)=1$ and $p+q=y$, then $(p, y)=1$ and $(q, y)=1$.
For instance, if $(p, y)=s$, then $p=p_1s$ and $y=y_1s$. 
It follows that $q=y-p=(y_1-q_1)s$.
It contradicts to $(p, q)=1$.

Consider Euler's  function $\phi(y)$.
It is not hard to prove that if $k<y$ and $k$ is  prime to $y$, then
$y-k$ is   prime to $y$ and  $y-k$ is  prime to $k$.
If $(y-k, y)=s$, then $y-k=r_1s$ and $y=r_2s$. It follows that $k=(r_2-r_1)s$.
It contradicts to $(k, y)=1$.

Therefore, we obtain  that the set of  integers not greater than and prime to $y$ are separated
 into the  pairs of coprime integers  
$(p,q)$ such that $p<q$ and $p+q= y$
 It follows that  $\phi(y)$ is even  and $\hat \psi(y) =\frac{1}{2} \phi(y)$.
From (\ref{psihatpsi}) we have
\begin{equation}
\psi(x) = \frac{1}{2} \sum\limits_{y=1}^{x} \phi(y).\notag
\end{equation}

Using (\ref{phi}), we obtain required formula  (\ref{psihatpsi2}).
\end{proof}

Denote by $N(L, \alpha)$ a number of simple closed geodesics of length not greater than $L$ 
on a regular tetraedron in Lobachevsky space with  the  plane angles of the faces equal to $\alpha$.
For each pair of coprime integers $(p,q)$, $p<q$  corresponds three  simple closed geodesics
 on a regular tetraedron in hyperbolic space.
 From Lemma \ref{condpq} we obtain
\begin{equation}
N(L, \alpha) =3\psi \left(  \frac{3}{4} 
\frac {  L -  2 \ln  \left(    \frac{2\pi^3 - \left( \pi - 3\alpha \right)^3  \left(1- \frac{4\alpha^2}{\pi^2}\right) } 
{2\pi^3 - \left( \pi - 3\alpha \right)^3  \left(1+ \frac{4\alpha^2}{\pi^2}\right) } \right)  } 
{    \ln \left(   \frac{2\pi^3 - \left( \pi - 3\alpha \right)^3 \left(1- \frac{\alpha^2}{\pi^2}\right) }
{2\pi^3 - \left( \pi - 3\alpha \right)^3\left(1+ \frac{\alpha^2}{\pi^2}\right) } \right) +
\ln \left(  \frac{3  \pi-3\alpha}{\pi + 3\alpha}   \right)   } 
+2  \right)  \notag
\end{equation}
Using  (\ref{psihatpsi2}),  we get 
\begin{eqnarray*}
&
 N(L, \alpha) =   \frac{27 L^2}
 {32 \pi^2 
 \left( 
  \ln \left(
     \frac{2\pi^3 - \left( \pi - 3\alpha \right)^3 \left(1- \frac{\alpha^2}{\pi^2}\right) }
	{2\pi^3 - \left( \pi - 3\alpha \right)^3\left(1+ \frac{\alpha^2}{\pi^2}\right) }
       \right) +
\ln \left( 
 \frac{3  \pi-3\alpha}{\pi + 3\alpha}  
 \right) 
 \right) ^2} +O(L\ln L),  \; {\text{при }}  L\rightarrow +\infty.
&
\end{eqnarray*}
Thus we proved follow theorem.
\begin{theorem**}
Let $N(L, \alpha) $ be a number of simple closed geodesics of length  not greater than $L$ 
on a regular tetraedron  with  plane angles of the faces equal to $\alpha$ in Lobachevsky space.
Then there exists the function  $c(\alpha)$ such that 
\begin{equation}\label{NLalpha}
 N(L, \alpha) =
c(\alpha) L^2 +O(L\ln L),
\end{equation}
where $O(L\ln L) \le CL\ln L$ when $L \rightarrow +\infty$,
\begin{equation}
 c(\alpha) =  \frac{27}{32 \pi^2} \frac{1}{ \left(  \ln \left(   \frac{2\pi^3 - \left( \pi - 3\alpha \right)^3 \left(1- \frac{\alpha^2}{\pi^2}\right) }
{2\pi^3 - \left( \pi - 3\alpha \right)^3\left(1+ \frac{\alpha^2}{\pi^2}\right) } \right) +
\ln \left(  \frac{3  \pi-3\alpha}{\pi + 3\alpha}   \right)  \right) ^2}, \notag
\end{equation}
\begin{equation}\label{cpi3}
\lim_{\alpha \rightarrow \frac{\pi}{3}} c(\alpha) = +\infty;
\; \; \; \; 
\lim_{\alpha \rightarrow 0} c(\alpha) = \frac{27}{32  (\ln 3)^2 \pi^2 }.
\end{equation}
\end{theorem**}

If  $\alpha$  goes to zero, then the vertices of the tetrahedron tends to infinity.
The limiting surface is non-compact surface with  complete regular Riemannian metric of constant negative curvature.
Note that the  limiting  surface is  homeomorphic to a sphere with four points at infinity.
Then the genus of imiting surface is  equal to zero.
I. Rivin show that  on such surface
the number of simple closed  geodesics of length no greater than $L$
is of order $L^2$.

From the equations (\ref{cpi3}) and (\ref{NLalpha}) we obtain  that if $\alpha$  goes to zero,
then the number  $N(L, \alpha)$  is asymptotic to $L^2$ when $L \rightarrow +\infty$.
Hence,   our results are consistent with the results of I. Rivin.

 {\it We are grateful to V. A. Gorkavyy for valuable discussions. }

  B.Verkin Institute for Low
Temperature Physics and Engineering of the National Academy of Sciences of Ukraine,
Kharkiv, 61103, Ukraine

\textsc {  \textit  {E-mail address: } } aborisenk@gmail.com, suhdaria0109@gmail.com

\end{document}